\documentclass{amsart}

    \usepackage[
    top    = 3cm,
    bottom = 3cm,
    left   = 3cm,
    right  = 3cm]{geometry}

\usepackage{stmaryrd}
\usepackage{amssymb}
\usepackage{mathrsfs}
\usepackage{xcolor}

\newtheorem{theorem}{Theorem}[section]
\newtheorem{lemma}[theorem]{Lemma}
\newtheorem{cor}[theorem]{Corollary}

\newtheorem*{thmA}{Theorem A}
\newtheorem*{thmB}{Theorem B}

\theoremstyle{definition}

\newtheorem{hyp}[theorem]{Hypothesis}

\theoremstyle{remark}
\newtheorem{remark}[theorem]{\bf Remark}

\numberwithin{equation}{section}

\newcommand{\fA}{{\mathfrak{A}}}
\newcommand{\fS}{{\mathfrak{S}}}

\newcommand{\bC}{{\mathbf{C}}}

\newcommand{\bF}{{\mathbf{F}}}

\newcommand{\Aut}{{\operatorname{Aut}}}
\newcommand{\Stab}{{\operatorname{Stab}}}
\newcommand{\IBr}{{\operatorname{IBr}}}
\newcommand{\Irr}{{\operatorname{Irr}}}

\newcommand{\Syl}{{\operatorname{Syl}}}

\newcommand{\dl}{{\operatorname{dl}}}
\newcommand{\cl}{{\operatorname{cl}}}

\newcommand{\Out}{{\operatorname{Out}}}

\newcommand{\clsize}{{\operatorname{clsize}}}
\newcommand{\cd}{{\operatorname{cd}}}

\let\nor=\triangleleft

\begin{document}

\title[On $p$-parts of Brauer character degrees and $p$-regular conjugacy class sizes]{On $p$-parts of Brauer character degrees and $p$-regular conjugacy class sizes of finite groups}

\author{Christine Bessenrodt and Yong Yang}

\address{Institut f\"ur Algebra, Zahlentheorie und Diskrete Mathematik, Leibniz Universit\"at Hannover, Welfengarten 1, D-30167 Hannover, Germany}
\makeatletter
\email{bessen@math.uni-hannover.de}
\makeatother

\address{Department of Mathematics, Texas State University, 601 University Drive, San Marcos, TX 78666, USA and Key Laboratory of Group and Graph Theories and Applications, Chongqing University of Arts and Sciences, Chongqing 402160, China.}

\makeatletter
\email{yang@txstate.edu}
\makeatother

\date{October 13, 2018}



\begin{abstract}
Let $G$ be a finite group, $p$ a prime, and $\IBr_p(G)$ the set of irreducible $p$-Brauer characters of $G$. Let $\bar e_p(G)$ be the largest integer such that $p^{\bar e_p(G)}$ divides $\chi(1)$ for some $\chi \in \IBr_p(G)$. We show that $|G:O_p(G)|_p \leq p^{k \bar e_p(G)}$ for an explicitly given constant $k$. We also study the analogous problem for the $p$-parts of the conjugacy class sizes of $p$-regular elements of finite groups.
\end{abstract}

\maketitle
\maketitle
\Large

\section{Introduction} \label{sec:introduction8}

It is a classic theme to study how arithmetic conditions on characters of a finite group affect the structure of the group. Some of the most important problems in the representation theory of finite groups deal with character degrees and prime numbers.

Let $G$ be a finite group and $P$ be a Sylow $p$-subgroup of $G$; it is reasonable to expect that the $p$-parts
of the degrees of irreducible characters of $G$ somehow restrict the structure of $P$.
The Ito-Michler theorem says that each  irreducible ordinary character degree is coprime to $p$ if and only if $G$ has a normal abelian Sylow $p$-subgroup, which of course implies that $|G:O_p(G)|_p=1$.

We write $e_p(G)$ to denote the exponent of the largest $p$-part of the degrees of the irreducible complex characters of $G$. Moret\'o \cite[Conjecture 4]{Moret1} conjectured that the largest character degree of $P$ is bounded by some function of $e_p(G)$.
For the case of solvable groups,
the conjecture was proved by Moret\'o and Wolf ~\cite{MOWOLF}, and the bounds have been improved by the second author in ~\cite{YY1} and ~\cite{YY5}. Recently, Lewis, Navarro and Wolf ~\cite{LNW} studied the special case when
$e_p(G)=1$, and showed that $|G:O_p(G)|_p \leq p^2$ when $G$ is solvable.
For $p>2$,
Lewis, Navarro, Tiep and Tong-Viet ~\cite{LNTH} also studied the case when
$e_p(G)=1$ for arbitrary finite groups.
The conjecture of Moret\'o was recently settled by Qian and the second author in ~\cite{YY25}.

It is natural to study the Brauer character degree analogue of the Ito-Michler theorem. This has been investigated by Michler ~\cite{MichlerSurvey} and Manz~\cite{Manz}. They showed that each irreducible Brauer character degree is coprime to $p$ if and only if $G$ has a normal Sylow $p$-subgroup, i.e., that $|G:O_p(G)|_p=1$. Thus we would like to ask the following question:

Let $\IBr_p(G)$ be the set of irreducible $p$-Brauer characters of $G$,
and $\bar e_p(G)$ be the largest integer such that $p^{\bar e_p(G)}$ divides $\chi(1)$ for some $\chi \in \IBr_p(G)$, then how does $\bar e_p(G)$ affect the structure of the Sylow $p$-subgroup of $G$? We show the following results as an effort to study this question. This could be viewed as a generalization of the Brauer character degree analogue of the Ito-Michler theorem. We remark that the case $\bar e_p(G)=1$ has been studied in ~\cite{LNTH}.

\begin{thmA}
Let $G$ be a finite group and $\bar e_p(G)$ be the largest integer such that $p^{\bar e_p(G)}$ divides $\chi(1)$ for some $\chi \in \IBr_p(G)$.
\begin{enumerate}
\item If $p \geq 5$, then $\log_p |G: O_p(G)|_p\leq 6.5 \,\bar e_p(G)$.
\item If $p=3$, then $\log_p |G: O_p(G)|_p\leq 20 \,\bar e_p(G)$.
\item If $p=2$, then $\log_p |G: O_p(G)|_p\leq 24 \,\bar e_p(G)$.
\end{enumerate}
\end{thmA}

\bigskip

As conjugacy classes are closely related to the irreducible characters, we could study related questions on conjugacy class sizes. The conjugacy class size analogues of $p$-Brauer character degrees are obviously the class sizes of the $p$-regular elements.
Similarly to the situation for $p$-Brauer character degrees, it is also reasonable to expect that the $p$-parts of the conjugacy class sizes of the $p$-regular elements somehow restrict the structure of $P$.

Let $\overline{ecl}_p(G)$ be the largest integer such that $p^{\overline{ecl}_p(G)}$ divides some $|C| \in \clsize_{p'}(G)$; we will show that $|G:O_p(G)|_p$ is also bounded by a function of $\overline{ecl}_p(G)$.

\begin{thmB}
Let $G$ be a finite group and let $p$ be a prime; let $P \in \Syl_p(G)$. Let $\overline{ecl}_p(G)$ be the largest integer such that $p^{\overline{ecl}_p(G)}$ divides some $|C| \in \clsize_{p'}(G)$.

\begin{enumerate}
\item If $p \geq 5$, then $\log_p |G: O_p(G)|_p\leq 6.5 \,\overline{ecl}_p(G)$.
\item If $p=3$, then $\log_p |G: O_p(G)|_p\leq 19 \,\overline{ecl}_p(G)$.
\item If $p=2$, then $\log_p |G: O_p(G)|_p\leq 17 \,\overline{ecl}_p(G)$.
\end{enumerate}
\end{thmB}

We notice that recently Tong-Viet has done some related work in finding various conditions on Brauer character degrees for a finite group to have a normal Sylow $p$-subgroup (see ~\cite{HTV}).

\section{Notation and preliminary results} \label{sec:Orbits}
We first fix some notation:
\begin{enumerate}

\item We use $\bF(G)$ to denote the Fitting subgroup of $G$. Let $\bF_i(G)$ be the $i$th ascending Fitting subgroup of $G$, i.e., $\bF_0(G)=1$, $\bF_1(G)=\bF(G)$ and $\bF_{i+1}(G)/\bF_i(G)=\bF(G/\bF_i(G))$.

\item We use $\bF^*(G)$ to denote the generalized Fitting subgroup of $G$.

\item Let $p$ be a prime number, we say that an element $x \in G$ is $p$-regular if the order of $x$ is not a multiple of $p$.

\item We use $\clsize_{p'}(G)$ to denote the set of conjugacy class sizes of $p$-regular elements of $G$.

\item We use $\cl(G)$ to denote the set of all the conjugacy classes of $G$, and we use $\cl_{p'}(G)$ to denote the set of all the conjugacy classes of $p$-regular elements of $G$.

\item We denote $\cd(G)=\{\chi(1)\ | \ \chi \in \Irr(G) \}$.

\item We use $\IBr_p(G)$ to denote the set of all the irreducible $p$-Brauer characters of $G$.

\item Let $\bar e_p(G)$ be the largest integer such that $p^{\bar e_p(G)}$ divides $\chi(1)$ for some $\chi \in \IBr_p(G)$.

\item Let $\overline{ecl}_p(G)$ be the largest integer such that $p^{\overline{ecl}_p(G)}$ divides some $|C| \in \clsize_{p'}(G)$.

\item We use the notation $\dl(G)$ for the derived length of a solvable group $G$.

\item If a group $G$ acts on a set $\Omega$ and $\omega$ is an element in $\Omega$, we will use the notation $\bC_G(\omega)$ to denote the stabilizer of the element $\omega$ under the action of $G$. In particular, if $\lambda$ is an irreducible character of a normal subgroup $N$ of $G$, then $\bC_G(\lambda)$ denotes the inertia group of $\lambda$ in $G$. Let $\Omega_1$ be a subset of $\Omega$, we use $\Stab_G{\Omega_1}$ to denote the stabilizer of $\Omega_1$ under the action of $G$ as a set (consider the induced action of $G$ on $\mathcal{P}(\Omega)$, the power set of $\Omega$).
\end{enumerate}

\bigskip

We need the following results about simple groups.

\begin{lemma}\label{simplecoprime}
Let $A$ act faithfully and coprimely on a non-abelian simple group $S$. Then $A$ has at least $2$ regular orbits on $\Irr(S)$.
\end{lemma}
\begin{proof}
This is ~\cite[Proposition 2.6]{MOTIEP}.
\end{proof}

\begin{lemma}\label{chardegsimple}
If $G$ is a non-abelian finite simple group, then $|\cd(G)| \geq 4$.
\end{lemma}
\begin{proof}
This follows from ~\cite[Theorem 12.15]{Isaacs/book}.
\end{proof}

\begin{lemma}\label{classsimple}
If $G$ is a non-abelian finite simple group, then $|cs(G)| \geq 4$.
\end{lemma}
\begin{proof}
This follows from ~\cite{Ito}.
\end{proof}

The main results are proved using an orbit theorem for $p$-solvable groups. This method provides a unified approach to the Brauer character degree and the $p$-regular class size version of the problem.

We now state the orbit theorem for $p$-solvable groups. This result has been proved in ~\cite{YY25} but the proof there has some glitch; we take the opportunity to provide a corrected proof here.

\begin{theorem} \label{charorbit}
Let $V\trianglelefteq G$, where $G/V$ is $p$-solvable for an odd prime $p$,
and  $V$ is a direct product of isomorphic non-abelian simple groups $S_1, \ldots, S_n$.
Suppose that $G$ acts transitively on the groups $S_1,\ldots,S_n$, and write $O=\bigcap_k N_G(S_k)$.
Then there exist nonprincipal $v_1$, $v_2$ and  $v_3  \in {\rm Irr}(V)$
of different degrees such that all Sylow $p$-subgroups of $\bC_G(v_j)$
are contained in $O$ for all $j=1,2,3$.
\end{theorem}

\begin{proof}

Clearly $O$ is normal in $G$ and
$G$ is  a transitive permutation group
on the set $\{S_1, \ldots, S_n\}$ with kernel $O$.
If $n=1$, then the required result follows by Lemma ~\ref{chardegsimple}.
Thus we may assume that $n >1$.
Let  $(\Delta_1, \dots, \Delta_m)$ be
a system of imprimitivity of  $G$  with maximal block-size $b$.
Then $(\Delta_1, \dots, \Delta_m)$ is a partition of $\{S_1, \ldots, S_n\}$ and each block
$\Delta_i$ has size $b$. Thus  $$1 \leq b <n; bm=n, m\geq 2.$$
Let $\Omega=\{\Delta_1, \ldots, \Delta_m\}$.
Then $G$ is a primitive permutation group of degree $m$ on the set $\Omega$.
Set
 $$J_i={\rm Stab}_{G}(\Delta_i),\,\, K =\bigcap_{1\leq i\leq m}  J_i,\,\,
 V_i=\prod_{S_t \in \Delta_i} S_t, i=1, \ldots,m.$$
Observe that $$J_i=N_G(V_i),$$
the groups $J_i$ are permutationally equivalent transitive groups of degree $b$,
and that  $K$ is a normal subgroup of $G$  and stabilizes
each of the blocks $\Delta_i$.
In particular,  $G/K$ is a primitive group
of degree $m$ acting upon the set $\Omega$.

Let us consider $\sigma\in \Irr(V_i)$.
We may view $\sigma$ as a character of $V$.
Note that if $\sigma$ is nonprincipal,
then
$\bC_G(\sigma)\leq J_i$ because $G$ acts transitively on $\Omega$,
and therefore $\bC_G(\sigma)=\bC_{J_i}(\sigma)$.

Let us consider $J_i$ and the action of $J_i$
on ${\rm Irr}(V_i)=\prod_{S_t \in \Delta_i} {\rm Irr}(S_t)$.
Since $G$ acts transitively on $\{S_1, \ldots, S_n\}$ and acts transitively on $\Omega$,
we see that  $J_i$ acts transitively on $\Delta_i$.
Write $$O_i=\bigcap_{t\in \Delta_i} N_{J_i}(S_t), i=1,\ldots, m.$$
Clearly $O=\bigcap_{i=1}^m O_i$.
Note that if $S_t\in \Delta_i$,
then $N_G(S_t)\leq J_i$ because $G$ acts transitively on $\Omega$.
Therefore $N_G(S_t)=N_{J_i}(S_t)$,
and this implies that $$O_i= \bigcap_{S_t\in \Delta_i} N_{G}(S_t).$$
Since $J_i< G$, by induction there exist nonprincipal $\theta_i, \lambda_i$,
and $\chi_i \in {\rm Irr}(V_i)$ of different degrees
such that all Sylow $p$-subgroups of $\bC_{J_i}(\theta_i), \bC_{J_i}(\lambda_i)$ and
$\bC_{J_i}(\chi_i)$ are contained in $O_i$, that is, all Sylow  $p$-subgroups of $\bC_{G}(\theta_i), \bC_{G}(\lambda_i)$,
$\bC_{G}(\chi_i)$  are contained in $O_i$.
Clearly we may choose $\theta_i$, $1 \leq i \leq m$, to be $G$-conjugate,
and we can do the same for $\lambda_i$ and $\chi_i$.
We may assume that $\theta_i(1)> \lambda_i(1)> \chi_i(1)$.

\medskip

We claim that there exist proper subsets $\Omega_1$ and $\Omega_2$ of $\Omega$ such that
$\Omega=\Omega_1\cup \Omega_2$, $\Omega_1\cap \Omega_2=\emptyset$,
and ${\rm Stab}_{G/K}(\Omega_1) \cap {\rm Stab}_{G/K}(\Omega_2)$ is a $p'$-group except for a few cases listed below.

(1) $|\Omega|=8$, $G/K \cong \text{A}\Gamma\text{L}(1,8)$.

(2) $|\Omega|=9$, $G/K \cong \text{AGL}(2,3)$ or $G/K \cong \text{ASL}(2,3)$.

To see the claim, we need to investigate
the action of $G/K$ on the power set $\mathcal{P}(\Omega)$ of $\Omega$.
Clearly we may assume that $p$ divides $|G/K|$.
Note that if $m\geq 5$, then ${\rm Alt}(m) \not \leq G/ K$
because  $G/K$ is $p$-solvable.
Note that if  $G/K$ has a regular orbit on $\mathcal{P}(\Omega)$,
then  there exists a (clearly proper) subset $\Omega_1$ of $\Omega$ such that
${\rm Stab}_{G/K}(\Omega_1)=1$,
thus $\Omega_1$ and $\Omega_2=\Omega-\Omega_1$ meet our requirement.
Hence we may assume $G$ has no regular orbit on $\mathcal{P}(\Omega)$.

Suppose that $G/K$ is solvable. By Gluck's result about solvable primitive permutations groups \cite{GluckPerm}, we see that there exists a partition $\Omega_1, \Omega_2$ of $\Omega$
such that ${\rm Stab}_{G/K}(\Omega_1)\cap {\rm Stab}_{G/K}(\Omega_2)$ is a $2$-group, except for the following cases:

(1) $n=8$, $G/K \cong \text{A}\Gamma\text{L}(1,8)$.

(2) $n=9$, $G/K \cong \text{AGL}(2,3)$ or $G/K \cong \text{ASL}(2,3)$.

Suppose that $G/K$ is nonsolvable.
By ~\cite[Theorem 2]{Seress}, $G/K$ is not $r$-solvable for any prime divisor $r$ of $|G/K|$,
we get a contradiction.

We first assume that there exist proper subsets $\Omega_1$ and $\Omega_2$ of $\Omega$ such that
$\Omega=\Omega_1\cup \Omega_2$, $\Omega_1\cap \Omega_2=\emptyset$,
and ${\rm Stab}_{G/K}(\Omega_1) \cap {\rm Stab}_{G/K}(\Omega_2)$ is a $p'$-group.

Assume that $\Omega_1=\{\Delta_1,\ldots, \Delta_s\}$,
$\Omega_2=\{\Delta_{s+1},  \ldots, \Delta_m\}$.
Set
$$v_1= \prod_{i=1}^{s} \theta_i \cdot \prod_{i=s+1}^{m} \lambda_i,\,\,
v_2= \prod_{i=1}^{s} \theta_i \cdot \prod_{i=s+1}^{m} \chi_i,\,\,
v_3= \prod_{i=1}^{s} \lambda_i \cdot \prod_{i=s+1}^{m}\chi_i.$$
Clearly, $v_1, v_2$ and $v_3$ have different degrees.
Let us investigate $\bC_G(v_1)$ and its Sylow $p$-subgroup $P$.
Since $G$ acts transitively on $\Omega$ and thus on $\{V_1, \ldots, V_m\}$,
we see that $\bC_G(v_1)\leq {\rm Stab}_G(\Omega_1)\cap {\rm Stab}_G(\Omega_2)$.
As $({\rm Stab}_G(\Omega_1)\cap {\rm Stab}_G(\Omega_2))/K$ is a $p'$-group by the claim,
it forces that
$$P\leq K\cap \bC_G(v_1)\cap P=\bC_K(v_1)\cap P.$$
Observing that all groups $V_i$ are normal in $K$, we have
$$\bC_K(v_1)=(\bigcap_{ i=1}^s \bC_K(\theta_i))\cap (\bigcap_{i=s+1}^m \bC_K(\lambda_i)).$$
We get the required result  that
$$P\leq (\bigcap_{ i=1}^s (\bC_K(\theta_i)\cap P))\cap (\bigcap_{i=s+1}^m (\bC_K(\lambda_i)\cap P))\leq
\bigcap_{i=1}^m O_i=O.$$
Similarly all Sylow $p$-subgroups of $\bC_G(v_2)$ and $\bC_G(v_3)$ are contained in $O$.

We next assume that $n=8$, and $G/K \cong$ A$\Gamma$L$(1,8)$. We set $\Omega_1=\{1,2,3\}$, $\Omega_2=\{4,5,6\}$, and $\Omega_3=\{7,8\}$. We see that ${\rm Stab}_{G/K}(\Omega_1) \cap {\rm Stab}_{G/K}(\Omega_2) \cap {\rm Stab}_{G/K}(\Omega_3)=1$.

Set
$$v_1= \prod_{i \in \Omega_1} \theta_i \cdot \prod_{i \in \Omega_2} \lambda_i \cdot \prod_{i \in \Omega_3} \chi_i,\,\,
v_2= \prod_{i \in \Omega_1} \theta_i \cdot \prod_{i \in \Omega_2} \chi_i \cdot \prod_{i \in \Omega_3} \lambda_i,\,\,
v_3= \prod_{i \in \Omega_1} \lambda_i \cdot \prod_{i \in \Omega_2} \chi_i \cdot \prod_{i \in \Omega_3} \theta_i.$$
Clearly, $v_1, v_2$ and $v_3$ have different degrees. Let us investigate $\bC_G(v_1)$ and its Sylow $p$-subgroup $P$.
Since $G$ acts transitively on $\Omega$ and thus on $\{V_1, \ldots, V_m\}$,
we see that $\bC_G(v_1)\leq {\rm Stab}_G(\Omega_1)\cap {\rm Stab}_G(\Omega_2) \cap {\rm Stab}_G(\Omega_3)$.
As $({\rm Stab}_G(\Omega_1)\cap {\rm Stab}_G(\Omega_2) \cap {\rm Stab}_G(\Omega_3))/K$ is a trivial group,
it forces that
$$P\leq K\cap \bC_G(v_1)\cap P=\bC_K(v_1)\cap P.$$
Observing that all groups $V_i$ are normal in $K$, we have
$$\bC_K(v_1)=(\bigcap_{i \in \Omega_1} \bC_K(\theta_i))\cap (\bigcap_{i \in \Omega_2} \bC_K(\lambda_i)) \cap (\bigcap_{i \in \Omega_3} \bC_K(\chi_i)).$$
We get the required result  that
$$P\leq (\bigcap_{i \in \Omega_1} (\bC_K(\theta_i)\cap P))\cap (\bigcap_{i \in \Omega_2} (\bC_K(\lambda_i)\cap P) \cap (\bigcap_{i \in \Omega_3} (\bC_K(\chi_i)\cap P))\leq
\bigcap_{i=1}^m O_i=O.$$
Similarly all Sylow $p$-subgroups of $\bC_G(v_2)$ and $\bC_G(v_3)$ are contained in $O$.

We finally assume that $n=9$, and $G/K \cong AGL(2,3)$ or $G/K \cong ASL(2,3)$. We set $\Omega_1=\{1,2,3,4\}$, $\Omega_2=\{5,6,7\}$, and $\Omega_3=\{8,9\}$. We see that ${\rm Stab}_{G/K}(\Omega_1) \cap {\rm Stab}_{G/K}(\Omega_2) \cap {\rm Stab}_{G/K}(\Omega_3)$ is a $2$-group.

Set
$$v_1= \prod_{i \in \Omega_1} \theta_i \cdot \prod_{i \in \Omega_2} \lambda_i \cdot \prod_{i \in \Omega_3} \chi_i,\,\,
v_2= \prod_{i \in \Omega_1} \lambda_i \cdot \prod_{i \in \Omega_2} \theta_i \cdot \prod_{i \in \Omega_3} \chi_i,\,\,
v_3= \prod_{i \in \Omega_1} \lambda_i \cdot \prod_{i \in \Omega_2} \chi_i \cdot \prod_{i \in \Omega_3} \theta_i.$$
Clearly, $v_1, v_2$ and $v_3$ have different degrees. Let us investigate $\bC_G(v_1)$ and its Sylow $p$-subgroup $P$.
Since $G$ acts transitively on $\Omega$ and thus on $\{V_1, \ldots, V_m\}$,
we see that $\bC_G(v_1)\leq {\rm Stab}_G(\Omega_1)\cap {\rm Stab}_G(\Omega_2) \cap {\rm Stab}_G(\Omega_3)$.
As $({\rm Stab}_G(\Omega_1)\cap {\rm Stab}_G(\Omega_2) \cap {\rm Stab}_G(\Omega_3))/K$ is a $2$-group,
it forces that
$$P\leq K\cap \bC_G(v_1)\cap P=\bC_K(v_1)\cap P.$$
Observe that all $V_i$s are normal in $K$, we have
$$\bC_K(v_1)=(\bigcap_{i \in \Omega_1} \bC_K(\theta_i))\cap (\bigcap_{i \in \Omega_2} \bC_K(\lambda_i)) \cap (\bigcap_{i \in \Omega_3} \bC_K(\chi_i)).$$
We get the required result  that
$$P\leq (\bigcap_{i \in \Omega_1} (\bC_K(\theta_i)\cap P))\cap (\bigcap_{i \in \Omega_2} (\bC_K(\lambda_i)\cap P) \cap (\bigcap_{i \in \Omega_3} (\bC_K(\chi_i)\cap P))\leq
\bigcap_{i=1}^m O_i=O.$$
Similarly all Sylow $p$-subgroups of $\bC_G(v_2)$ and $\bC_G(v_3)$ are contained in $O$.
\end{proof}

\section{On $p$-parts of $p$-Brauer character degrees} \label{Char p part of G/F(G)}

It is a fundamental fact in block theory that if an ordinary irreducible character $\chi$ is such that $\chi(1)_p=|G|_p$, for a prime $p$, then its reduction modulo $p$ gives an irreducible Brauer character of the same degree. Hence then $e_p(G)\le \bar e_p(G)$, and the bounds obtained with respect to ordinary characters still hold in the case of $p$-Brauer characters.

For the solvable case, the problems in this paper have been studied in ~\cite{MO1} and certain bounds were obtained; more explicitly, it was shown that for a finite solvable group $G$ with $O_p(G)=1$, $\log_p |G|_p \leq 96 \bar e_p(G)$ and $\log_p |G|_p) \leq 683 \overline{ecl}_p(G)$.
We greatly improve those bounds, and we will obtain corresponding results for arbitrary finite groups.

We first note that if $N$ is a normal subgroup of $G$, then it is easy to see that $\bar e_p (G/N) \leq \bar e_p(G)$ and $\bar e_p (N) \leq \bar e_p(G)$. We shall use this fact freely in the following arguments.

\smallskip

The following lemma is due to Martin Isaacs ~\cite{IMI2}.

\begin{lemma}\label{coprimeaction}
Let $P$ be a nontrivial $p$-group that acts faithfully on a group $H$, where $|H|$ is not divisible by $p$. Then there exists an element $x \in H$ such that $|\bC_P (x)| \leq  |P|^{1/2}$.
\end{lemma}


\subsection{The solvable case}

\begin{theorem} \label{solvableepboundp5}
Let $G$ be a finite solvable group with $O_p(G)=1$, where $p \geq 5$;
set $n=\bar e_p(G)$. Then $|G|_p \leq p^{2.5n}$.
\end{theorem}

\begin{proof}
Let $|G|_p=p^a$. By \cite{YY5}, the group $G$ has a $p$-block of defect $d\le \frac 35 a$.
Since $a-d \le n$ (see \cite[Section 15]{Isaacs/book}), we obtain $a\le \frac 52 n$. Hence the claim holds.
\end{proof}

\begin{remark}
For $G$ a group of odd order with $O_p(G)=1$, Espuelas and Navarro have shown in \cite{EspuelasNavarro} that there is in fact a $p$-block of defect $d \le \lfloor a/2 \rfloor$ (and this bound is best possible).
Using the same argument (and notation) as above, we then obtain the better bound
$|G|_p \le p^{2n}$ in Theorem~\ref{solvableepboundp5}. Already in \cite{EspuelasNavarro} the question is posed whether for finite groups with $O_p(G)=1$ and $p\ge 5$, such $p$-blocks of small defect always exist; clearly, this would then also give a better bound in Theorem~A, for $p\ge 5$.
It was already noticed in \cite{EspuelasNavarro} that for $p=2$
for example the group $G=\fA_7$ has no 2-block of the desired small defect 1;
note that we still have $|G|_2\le 2^{2n}$ in this case.
However,  the example $G=M_{22}$ (discussed later) shows that
for $p=2$ the bound $|G|_2\le 2^{2n}$ does not hold in general;
there may still be room to improve the bounds given in Theorem~A, though.
\end{remark}


\begin{theorem} \label{solvableepbound}
Let $G$ be a finite solvable group with $O_p(G)=1$ and set $n=\bar e_p(G)$. Then $|G|_p \leq p^{15n}$ if $p=2$ or $p=3$.
\end{theorem}
\begin{proof}
By Gasch\"utz's theorem, $G/\bF(G)$ acts faithfully and completely reducibly on $\Irr(\bF(G)/\Phi(G))$. Since $p \nmid |\bF(G)/\Phi(G)|$, $\Irr(\bF(G)/\Phi(G))=\IBr(\bF(G)/\Phi(G))$. It follows from ~\cite[Theorem 3.3]{YY1} that there exists $\lambda \in \IBr(\bF(G)/\Phi(G))$ such that $T = \bC_{G}(\lambda) \leq \bF_8(G)$.

Let $K_{i+1}=\bF_{i+1}(G)/\bF_i(G)$ and let $K_{i+1, p}$ be the Sylow $p$-subgroup of $K_{i+1}$ for all $i \geq 1$.
We know that $K_{i+1, p}$ acts faithfully and completely reducibly on $K_i/\Phi(G/\bF_{i-1}(G))$. It is clear that we may write $K_i/\Phi(G/\bF_{i-1}(G))=V_{i1}+V_{i2}$ where $V_{i1}$ is the $p$-part of $K_i/\Phi(G/\bF_{i-1}(G))$ and $V_{i2}$ is the $p'$-part of $K_i/\Phi(G/\bF_{i-1}(G))$ for all $i \geq 1$.

We observe that $K_{i+1, p}$ acts faithfully and completely reducibly on $\Irr(V_{i2})$ for all $i \geq 1$. Since $\IBr(V_{i2})=\Irr(V_{i2})$, we have $|K_{i+1, p}| \leq p^{2n}$ by Lemma ~\ref{coprimeaction}.

Next, we show that $|G : T|_{p} \leq p^n$.

Take $\chi\in \IBr(G)$ lying over $\lambda$. Then $|G : T|_p$ divides $\chi(1)$, which is at most $p^n$. 


We know from before that $|K_{i,p}| \leq p^{2n}$ for $2 \leq i \leq 8$. This implies that $|G|_p \leq (p^{2n})^7 \cdot p^n = p^{15 n}$.
\end{proof}


\subsection{The $p$-solvable case}

\-\smallskip

We now obtain bounds for $p$-solvable groups and then extend those to arbitrary groups.

\begin{theorem}\label{charppart}
Let $G$ be a $p$-solvable group for an odd prime $p$.
Assume that $G$ has no  nontrivial solvable normal subgroup.
Then there exists $\chi\in \IBr_p(G)$ such that $\chi(1)_p \geq \sqrt{|G|_p}$.
\end{theorem}
\begin{proof}
Since $G$ has no nontrivial solvable normal subgroup,
the socle $L$ of $G$ can been written as
$L= L_1\times \cdots  \times L_n$,
where  $L_i=S_{i1} \times \cdots \times S_{it_i}$  is minimal normal in $G$,
and $S_{i1}, \ldots, S_{i t_i}$ are isomorphic to a nonabelian simple group $S_i$.

We observe that since $G$ is $p$-solvable, $p \nmid |L|$. Thus $\IBr_p(L_i)=\Irr(L_i)$ and $\IBr_p(L)=\Irr(L)$.

Write $O_i=\bigcap_{j=1}^{t_i} N_G (S_{ij})$ and $O=\bigcap_{i=1}^n O_i$.
Clearly $O$ and all $O_i$ are normal in $G$, all $S_{ij}$ are normal in $O_i$ and $O$.
Repeatedly using Dedekind's Modular Law, we have that
$$L= \bigcap_{i=1}^n\bigcap_{j=1}^{t_i} S_{ij}\bC_G(S_{ij}).$$
This implies that
$$O/L=O/\bigcap_{i=1}^n\bigcap_{j=1}^{t_i} S_{ij}\bC_G(S_{ij})\lesssim
\prod_{i, j} N_G(S_{ij})/(S_{ij}\bC_G(S_{ij}))\lesssim \prod_{i,j} \Out(S_{ij}).$$
Since all $S_{ij}$ are $p$-solvable, $\Out(S_{ij})$ has a normal cyclic Sylow $p$-subgroup (for example, ~\cite[Lemma 2.3(ii)]{LNTH}).
Thus $O/L$ has a normal and abelian Sylow $p$-subgroup.

By Lemma ~\ref{simplecoprime}, it is easy to find an irreducible character $\mu$ of $L$ such that
$\bC_O(\mu)$ is a $p'$-group.
Hence there exists an irreducible constituent $\chi_1$ of $\mu^G$ such that
$$\chi_1(1)_p \geq |O|_p.$$
Also, by Theorem ~\ref{charorbit}, we may find $\lambda_i\in \IBr_p(L_i)$ such that ${\bC}_G(\lambda_i)\leq O_i$ for each $i$.
Set $\lambda=\prod_i\lambda_i$ and let $\chi_2$ be an irreducible constituent
of $\lambda^G$.
Since all $L_i$ are normal in $G$, we have
$$\bC_G(\lambda)=\bigcap_i \bC_G(\lambda_i)\leq \bigcap_i O_i=O.$$
This implies that $$\chi_2(1)_p\geq |G/O|_p.$$
Thus there exists $\chi\in \{\chi_1, \chi_2\}$ such that
$\chi(1)_p\geq \sqrt{|G|_p}$.
\end{proof}

\subsection{The general case}

\-\smallskip

For a group $G$, let $b(G)$ denote the largest degree of an irreducible character of $G$.

\begin{lemma}\label{chardegreeside}
Let $G$ be a finite group, $P\in \Syl_p(G)$ and $\bar P=P/O_p(G)$; set $n=\bar e_p(G)$.
Assume that $|G: O_p(G)|_p\leq p^{k n}$. Then $b(\bar P)\leq p^{{k n}/2}$ and $\dl(\bar P) \leq 4 + \log_2 n + \log_2 k$.
\end{lemma}
\begin{proof}
Clearly, $b(\bar P) \leq |\bar P|^{1/2} \leq p^{{k n}/ 2}$.

By ~\cite[Theorem 12.26]{Isaacs/book} and the nilpotency of $\bar P$, we have that $\bar P$ has an abelian subgroup $B$ of index at most $b(\bar P)^4$. By ~\cite[Theorem 5.1]{Podoski}, we deduce that $\bar P$ has a normal abelian subgroup $A$ of index at most $|\bar P:B|^2$. Thus, $|\bar P:A| \leq |\bar P:B|^2 \leq b(\bar P)^{8s}$, where $b(\bar P)=p^s$. By ~\cite[Satz III.2.12]{Huppert1}, $\dl(\bar P/A) \leq 1+\log_2(8s)$ and so $\dl(\bar P) \leq 2+ \log_2(8s)=5+\log_2(s)$. Since $s$ is at most ${k n}/ 2$, we have $\dl(\bar P) \leq 4 + \log_2 n + \log_2 k$.
\end{proof}

\begin{theorem}\label{chardegreepsolvablebound}
Let $G$ be a finite $p$-solvable group for an odd prime $p$, $P \in \Syl_p(G)$ and $\bar P=P/O_p(G)$; set $n=\bar e_p(G)$.
We set $k=4.5$ if $p\geq 5$, and $k=17$ if $p=3$.
Then $|G: O_p(G)|_p \leq p^{k n}$, $b(\bar P)\leq p^{{k n}/ 2}$, and $\dl(\bar P) \leq 4 + \log_2 n + \log_2 k$.
\end{theorem}

\begin{proof}
We first prove the assertion in the case when $p\geq 5$. In view of Lemma ~\ref{chardegreeside}, we only need to show that $|G: O_p(G)|_p\leq p^{4.5 n}$.

Let $T$ be the maximal normal solvable subgroup of $G$.
Since $O_p(G) \leq T$, $O_p(T)=O_p(G)$.
Since $T \nor G$, $p^{n+1}$ does not divide $\lambda(1)$ for all $\lambda \in \IBr_p(T)$.
Thus by Theorem ~\ref{solvableepboundp5}, $|T: O_p(G)|_p\leq p^{2.5n}$.

Let $\tilde G=G/T$ and $\bar G=\tilde G / \bF^*(\tilde G)$. It is clear that $\bF^*(\tilde G)$ is a direct product of finite non-abelian simple groups. Since $\tilde G$ is $p$-solvable, $p \nmid |\bF^*(\tilde G)|$.

By Theorem ~\ref{charppart}, $|\bar G|_p \leq p^{2n}$, and we are done in this case.

We now consider the case when $p=3$ and we only need to show that $|G: O_p(G)|_p\leq p^{17 n}$ in view of Lemma ~\ref{chardegreeside}. The proof is similar to the previous case when $p \geq 5$ but using Theorem ~\ref{solvableepbound} instead of Theorem ~\ref{solvableepboundp5}.
\end{proof}


By the work of \cite{Gagola}, and stated explicitly in ~\cite[Lemma 3.1]{LNTH},
we have the following result that is used in both the character context as well as the context of conjugacy classes:
\begin{lemma}\label{simplep}
Let $S$ be a finite non-abelian simple group and let $p$ be a prime dividing $|S|$. Then $|S|_p > |\Out(S)|_p$.
\end{lemma}

In dealing with the simple groups, we need the following result which completes \cite[Theorem 2.5]{HTV} in that the remaining cases of alternating groups ($\fA_n$
for $n\in \{22,24,26\}$) are treated,
and it is a slight correction as
the exception in the case of $\fA_7$ at $p=2$ was overlooked.

\begin{theorem} \label{simplepcharwithspecial}
Let $S$ be a finite non-abelian simple group,
and let $p$ be a prime divisor of $|S|$.
Then there exists $\phi \in \IBr_p(S)$
such that
$$|\Aut(S)|_p < \phi(1)_p^2 $$
except in the following cases:

- $p=2$, $S=M_{22}$, then $|\Aut(S)|_2=2^8$, and $\bar e_2(S)=1$;

- $p=2$, $S=\fA_7$, then $|\Aut(S)|_2=2^4$,
and $\bar e_2(S)=2$;

- $p=3$, $S=\fA_7$, then $|\Aut(S)|_3=3^2$ and $\bar e_3(S)=1$.
\end{theorem}

\begin{proof}
The precise statements in the listed exceptional cases are checked using the
information on Brauer characters provided in tables coming from GAP \cite{GAP}.
If we are not in one of these cases,
\cite[Theorem 2.5]{HTV} (in the corrected version,
including the exception for $\fA_7$ at $p=2$)
tells us that there are possibly only the cases of $S=\fA_n$ with $n\in \{22,24,26\}$
at $p=2$ where the desired inequality might not hold.

For $n=22, 24$ and $26$, we have $|\Aut(S)|_2=2^{19},2^{22}$ and $2^{23}$,
respectively;
in these cases, the 2-Brauer character tables are not available,
and using a similar argument as in \cite{HTV} for finding a
suitable Brauer character
in a 2-block of smallest defect is not strong enough.
So we have to use other methods to find $\phi\in \IBr_2(S)$ such that $\phi(1)_2$ is large.

We consider the Specht modules $S^\lambda$
of $\fS_n$ labelled by
the partitions $(10,7,4,1)$ of $22$,
$(14,7,2,1)$ of $24$, and $(14,7,4,1)$ of $26$;
the 2-powers in the degrees are $2^{13}$, $2^{12}$ and $2^{14}$,
respectively, by the hook formula.
By the Carter criterion \cite[24.9]{James-book},
in all three cases the 2-modular reduction is the
corresponding irreducible module $D^\lambda$.
Restricting these modules to $\fA_n$ gives irreducible modules for $\fA_n$
by Benson's criterion \cite{Benson}.
Hence the 2-powers in the degrees of the corresponding 2-Brauer characters are
sufficiently large, as required.
\end{proof}

\begin{cor} \label{simplepchar}
Let $S$ be a finite non-abelian simple group,
and let $p$ be a prime divisor of $|S|$.
Then there exists $\phi \in \IBr_p(S)$
such that $|\Aut(S)|_p < \phi(1)_p^2$ if $p \geq 5$, $|\Aut(S)|_p < \phi(1)_p^3$ if $p=3$, and $|\Aut(S)|_p < \phi(1)_p^9$ if $p=2$.
\end{cor}
\begin{proof}
This is a direct corollary of Theorem ~\ref{simplepcharwithspecial}.
\end{proof}

\begin{hyp} \label{hypothesis}
Let $p$ be a prime and let $N = W_1 \times \cdots \times W_s$ be a normal subgroup of a finite group $G$ with the following assumptions: $\bC_G(N) = 1$; every $W_i$, $1 \leq i \leq s$, is a non-abelian simple group of order divisible by $p$.
\end{hyp}

\begin{lemma}\label{qianinductionchar}
Let  $G$, $N$, $p$ be as in Hypothesis ~\ref{hypothesis}. If there exists $\phi_i \in \IBr_p(W_i)$ such that $|\Aut(W_i)|_p< \phi_i(1)_p^k$ for every $i=1,\dots,s$, then there exists $\phi \in \IBr_p(N)$ such that $|G|_p<\phi(1)_p^k$.
\end{lemma}
\begin{proof}
The proof is the same as ~\cite[Lemma 2.6]{QIAN}.
\end{proof}

\begin{theorem}\label{chardegreegeneralbound}
Let $G$ be a finite group, $p$ be a prime, $P \in \Syl_p(G)$ and $\bar P=P/O_p(G)$;
set $n=\bar e_p(G)$.
We set $k=6.5$ if $p\geq 5$, $k=20$ if $p=3$, and $k=24$ if $p=2$. Then $|G: O_p(G)|_p\leq p^{kn}$, $b(\bar P)\leq p^{{k n}/ 2}$, and $\dl(\bar P) \leq 4 + \log_2 n + \log_2 k$.

\end{theorem}
\begin{proof}
Let $T$ be the maximal normal $p$-solvable subgroup of $G$. Since $O_p(G) \leq T$, $O_p(T)=O_p(G)$. Since $T \nor G$, $p^{n+1}$ does not divide $\lambda(1)$, for all $\lambda \in \IBr_p(T)$.

If $p \geq 5$, then $|T: O_p(G)|_p \leq p^{4.5n}$ by Theorem ~\ref{chardegreepsolvablebound}. If $p = 3$, then $|T: O_p(G)|_p \leq p^{17n}$ by Theorem ~\ref{chardegreepsolvablebound}. If $p =2$, then $|T: O_p(G)|_p \leq p^{15n}$ by Theorem ~\ref{solvableepbound}.

We now consider $\bar G=G/T$, we know that $\bF^*(\bar G)$ is a direct product of non-abelian simple groups, where $p$ divides the order of each of them.

Since $\bar G$ and $\bF^*(\bar G)$ satisfy Hypothesis ~\ref{hypothesis}, by Lemma ~\ref{qianinductionchar} and Corollary ~\ref{simplepchar}, we have that $|\bar G|_p \leq p^{2n}$ if $p \geq 5$, $|\bar G|_p \leq p^{3n}$ if $p=3$, and $|\bar G|_p \leq p^{9n}$ if $p=2$.

Thus, we have,
\begin{enumerate}
\item $|G:O_p(G)|_p \leq |G:T|_p |T: O_p(G)|_p  \leq p^{6.5 n}$ if $p \geq 5$.
\item $|G:O_p(G)|_p \leq |G:T|_p |T: O_p(G)|_p  \leq p^{20 n}$ if $p=3$.
\item $|G:O_p(G)|_p \leq |G:T|_p |T: O_p(G)|_p  \leq p^{24 n}$ if $p=2$.
\end{enumerate}

The bounds for $b(\bar P)$ and $\dl(\bar P)$ follow from Lemma ~\ref{chardegreeside}.
\end{proof}

\section{On $p$-parts of $p$-regular conjugacy class sizes} \label{conj p part of G/F(G)}

We now start to prove results related to the $p$-parts of $p$-regular conjugacy class sizes.

With respect to the $p$-regular class size version of the problem, we make the following observations. We will use the following results very often in the proofs so we state them here.

\begin{lemma}\label{lem1}
Let $N$ be a normal subgroup of $G$. Then
\begin{enumerate}
\item If $x \in N$, $|x^N|$ divides $|x^G|$.
\item If $x \in G$, $|(xN)^{G/N}|$ divides $|x^G|$.
\end{enumerate}
\end{lemma}

\begin{remark}
We first observe that the condition $p^k$ does not divide $|x^G|$ for every $p$-regular element $x \in G$ is inherited by all the normal subgroups of $G$ and all the quotient groups of $G$. Since the normal subgroups case easily follows from Lemma ~\ref{lem1}(1), we will just explain for the quotient groups. Let $N \nor G$, and $T$ be a $p$-regular class of $G/N$ £¬then we have a $p$-regular element $xN \in G/N$ such that $T= (xN)^{G/N}$. We may write $x= yz$, where $y$ is a $p'$-element, $z$ is a $p$-element and $yz=zy$. Let $H=\langle x \rangle N$, we know that $|H/N|$ is a $p'$ number, and thus $z\in N$. We have $xN=yN$, and $T=(yN)^{G/N}$. We have that $|T| \mid |y^G|$ and the result follows.
\end{remark}

\begin{theorem}\label{conjugacyboundsolvable}
Let $G$ be a solvable group with $O_p(G) = 1$, and let $P\in \Syl_p(G)$. Set $n=\overline{ecl}_p(G)$. Then $|G|_p \leq p^{15n}$ if $p=2$ or $p=3$. In particular, $e_p(G) \leq 15 n$, $b(P) \leq p^{7.5n}$, and $\dl(P)$ is bounded by a logarithmic function of $n$.
\end{theorem}
\begin{proof}
By Gasch\"utz's theorem, $G/\bF(G)$ acts faithfully and completely reducibly on $\bF(G)/\Phi(G)$. Since $p \nmid |\bF(G)/\Phi(G)|$, every element in $\bF(G)/\Phi(G)$ is a $p'$-element. It follows from ~\cite[Theorem 3.3]{YY1} that there exists $x \in \bF(G)/\Phi(G)$ such that $T = \bC_{G}(x) \leq \bF_8(G)$.

Let $K_{i+1}=\bF_{i+1}(G)/\bF_i(G)$ and let $K_{i+1, p}$ be the Sylow $p$-subgroup of $K_{i+1}$ for all $i \geq 1$.
We know that $K_{i+1, p}$ acts faithfully and completely reducibly on $K_i/\Phi(G/\bF_{i-1}(G))$. It is clear that we may write $K_i/\Phi(G/\bF_{i-1}(G))=V_{i1}+V_{i2}$ where $V_{i1}$ is the $p$-part of $K_i/\Phi(G/\bF_{i-1}(G))$ and $V_{i2}$ is the $p'$-part of $K_i/\Phi(G/\bF_{i-1}(G))$ for all $i \geq 1$.

We observe that $K_{i+1, p}$ acts faithfully and completely reducibly on $V_{i2}$ for all $i \geq 1$. Since $p \nmid |V_{i2}|$, every element in $V_{i2}$ is a $p'$-element. We have $|K_{i+1, p}| \leq p^{2n}$ by Lemma ~\ref{coprimeaction}.

Next, we show that $|G : T|_{p} \leq p^n$.


We now consider $|x^G|$; clearly $|G : T|_p$ divides $|x^G|$, hence is at most $p^n$. 

We know from before that $|K_{i,p}| \leq p^{2n}$ for $2 \leq i \leq 8$. This implies that $|G|_p \leq (p^{2n})^7 \cdot p^n = p^{15 n}$.
\end{proof}

\begin{theorem}\label{conjugacyboundsolvablep5}
Let $G$ be a solvable group with $O_p(G) = 1$ where $p \geq 5$ is a prime, and let
$P\in \Syl_p(G)$; set $n=\overline{ecl}_p(G)$. Then $|G|_p  \leq p^{2.5 n}$. In particular, $e_p(G) \leq 2.5 n$, $b(P) \leq p^{1.25n}$, and $\dl(P)$ is bounded by a logarithmic function of $n$.
\end{theorem}

\begin{proof}
Let $|G|_p=p^a$. By \cite{YY5}, the group $G$ has a $p$-block of defect $d\le \frac 35 a$.
Now $G$ has a $p$-regular element $x\in G$ such that $|C_G(x)|_p=p^d$
(see \cite[Section 15]{Isaacs/book}).
Hence $|x^G|_p=p^{a-d}$, which implies that $a-d \le n$, and thus
$a\le \frac 52 n$.
\end{proof}

We now state the class size version of Theorem ~\ref{charorbit}.

\begin{theorem} \label{classorbit}
Let $V\trianglelefteq G$, where $G/V$ is $p$-solvable for an odd prime $p$,
and  $V$ is a direct product of isomorphic non-abelian simple groups $S_1, \ldots, S_n$.
Suppose that $G$ acts transitively on the groups
$S_1, \ldots, S_n$, and write $O=\bigcap_k N_G(S_k)$.
Then there exist nonidentity $v_1$, $v_2$ and  $v_3  \in \cl(V)$
of different sizes such that all Sylow $p$-subgroups of $\bC_G(v_j)$
are contained in $O$ for all $j=1,2,3$.
\end{theorem}
\begin{proof}
The proof is similar to the proof of Theorem ~\ref{charorbit} but using Lemma ~\ref{classsimple} instead of Lemma ~\ref{chardegsimple}.
\end{proof}

We now prove the conjugacy class analogues of Theorem ~\ref{charppart} and Theorem ~\ref{chardegreepsolvablebound}.

\begin{theorem}\label{classppart}
Let $G$ be a $p$-solvable group for an odd prime $p$.
Assume that $G$ has no nontrivial solvable normal subgroup.
Then there exists $C \in \cl_{p'}(G)$ such that $|C|_p \geq \sqrt{|G|_p}$.
\end{theorem}
\begin{proof}
The proof is similar to the proof of Theorem ~\ref{charppart} but using Theorem ~\ref{classorbit} instead of Theorem ~\ref{charorbit}.
\end{proof}

\begin{lemma} \label{simplepconjugacy}
Let $S$ be a finite non-abelian simple group and $p \geq 3$ be a prime divisor of $|S|$, then there exists $C \in \cl_{p'}(S)$ such that $|\Aut(S)|_p<|C|_p^2$.
\end{lemma}
\begin{proof}
For the simple groups of Lie type and any prime $p$, or the alternating groups and $p \geq 5$, there is always a $p$-block of defect $0$.
Hence there is a $p$-regular element $x\in G$ such that
$|C_G(x)|_p=1$, and thus $|x^G|_p=|G|_p$.
Then the result follows from Lemma ~\ref{simplep}.

Thus one only needs to consider the alternating groups and $p=3$.

 First assume that $n$ is odd. If $\alpha$ is an $n$-cycle, then $\alpha \in \fA_n$ and  $|\cl_{\fA_n}(\alpha)|=\frac 1 2  (n-1)!$.
 If $\beta$ is an $(n-2)$-cycle, then $\beta \in \fA_n$ and $|\cl_{\fA_n}(\beta)|=  n!/((n-2)2)$. Now if $3 \nmid n$, then the class of $\alpha$ satisfies the condition. If $3 \mid n$, then $3 \nmid n-2$ and the class of $\beta$ satisfies the condition.

Now let $n$ be even.
 If $\alpha$ is an $(n-1)$-cycle, then $\alpha \in \fA_n$ and $|\cl_{\fA_n}(\alpha)|=\frac 1 2 \cdot \frac {n!} {n-1}$.
  If $\beta$ is an $(n-3)$-cycle, then $\beta \in \fA_n$ and $|\cl_{\fA_n}(\beta)|=\frac {n!} {(n-3) \cdot 6}$.
  Now if $3 \nmid n-1$, then the class of $\alpha$ satisfy the condition.
  If $3 \mid n-1$, then $3 \nmid n-3$ and the class of $\beta$ satisfies the condition.

For sporadic groups, the result can be checked by using ~\cite{CFSG}.
\end{proof}

Given a group $G$, we write $b^*(G)$ to denote the largest size of the conjugacy classes of $G$.

\begin{lemma}\label{conjugacyside}
Let $G$ be a finite group, $P\in \Syl_p(G)$ and $\bar P=P/O_p(G)$; set $n=\overline{ecl}_p(G)$.
Assume that $|G: O_p(G)|_p\leq p^{k n}$. Then $b^*(\bar P)\leq p^{k n}$, and $|\bar P'| \leq p^{kn(kn+1)/2}$.
\end{lemma}
\begin{proof}
It is clear that for $x \in \bar P$, we have $|x^{\bar P}|=|\bar P:\bC_{\bar P}(x)| \leq p^{k n}$.

To obtain the bounds for the order of $\bar P'$ it suffices to apply a theorem of Vaughan-Lee \cite[Theorem VIII.9.12]{Huppert2}.
\end{proof}

\begin{theorem}\label{conjugacyboundpsolvable}
Let $G$ be a finite $p$-solvable group for an odd prime $p$, $P \in \Syl_p(G)$, $\bar P=P/O_p(G)$; set $n=\overline{ecl}_p(G)$.
Then there exists a constant $k$ such that
$|G: O_p(G)|_p \leq p^{k n}$, $b^*(\bar P)\leq p^{k n}$, and $|\bar P'| \leq p^{k n(kn+1)/2}$ where $k=4.5$ if $p\geq 5$, and $k=17$ if $p=3$.
\end{theorem}
\begin{proof}
This is the class size version of Theorem ~\ref{chardegreepsolvablebound}, and the proof is similar. We first obtain the bound for $|G: O_p(G)|_p$, and then apply Lemma ~\ref{conjugacyside} to obtain the other parts.
\end{proof}

\begin{lemma}\label{qianinductionconj}
Let  $G$, $N$, $p$ be as in Hypothesis ~\ref{hypothesis}. If there exists $C_i \in \cl_{p'}(W_i)$ such that $|\Aut(W_i)|_p < |C_i|_p^k$ for every $i=1,\dots,s$, then there exists $C \in \cl_{p'}(N)$ such that $|G|_p<|C|_p^k$.
\end{lemma}
\begin{proof}
The proof is the same as that of ~\cite[Lemma 2.6]{QIAN}.
\end{proof}

\begin{theorem}\label{conjugacygeneralbound}
Let $G$ be a finite group, $p$ a prime, $P \in \Syl_p(G)$ and $\bar P=P/O_p(G)$;
set $n=\overline{ecl}_p(G)$.
We set $k=6.5$ if $p\geq 5$, $k=19$ if $p=3$, and $k=17$ if $p=2$. Then $|G: O_p(G)|_p \leq p^{kn}$, $b^*(\bar P)\leq p^{k n}$, and $|\bar P'| \leq p^{k n(k n+1)/2}$.
\end{theorem}
\begin{proof}
Let $T$ be the maximal normal $p$-solvable subgroup of $G$. Since $O_p(G) \leq T$, $O_p(T)=O_p(G)$. Since $T \nor G$, $p^{n+1}$ does not divide $|C|$ for all $C \in \cl_{p'}(T)$.

If $p \geq 5$, then $|T: O_p(G)|_p \leq p^{4.5n}$ by Theorem ~\ref{conjugacyboundpsolvable}. If $p=3$, then $|T: O_p(G)|_p \leq p^{17n}$ by Theorem ~\ref{conjugacyboundpsolvable}. If $p=2$, then $|T: O_p(G)|_p \leq p^{15n}$ by Theorem ~\ref{conjugacyboundsolvable}.

We now consider $\bar G=G/T$, we know that $\bF^*(\bar G)$ is a direct product of non-abelian simple groups, where $p$ divides the order of each of them.

Since $\bar G$ and $\bF^*(\bar G)$ satisfy Hypothesis ~\ref{hypothesis}, by Lemma ~\ref{qianinductionconj} and Lemma ~\ref{simplepconjugacy}, we have that $|\bar G|_p \leq p^{2n}$.

Thus, we have,
\begin{enumerate}
\item $|G:O_p(G)|_p \leq |G:T|_p |T: O_p(G)|_p  \leq p^{6.5 n}$ if $p \geq 5$.
\item $|G:O_p(G)|_p \leq |G:T|_p |T: O_p(G)|_p  \leq p^{19 n}$ if $p=3$.
\item $|G:O_p(G)|_p \leq |G:T|_p |T: O_p(G)|_p  \leq p^{17 n}$ if $p=2$.
\end{enumerate}

The bounds for $b^*(\bar P)$ and $|\bar P'|$ follow from Lemma ~\ref{conjugacyside}.
\end{proof}

\section{Acknowledgement} \label{sec:Acknowledgement}
This work was partially supported by the NSFC (No 11671063), and a grant from the Simons Foundation (No 499532, YY).


\end{document}